\documentclass[final,reqno,twoside,a4paper,12pt]{amsart}
\usepackage{lhtools}
\usepackage{lhmath62,lhthm10-REAR}

\usepackage{upref,amsmath,amssymb,amsthm,mathrsfs}
\usepackage{pgf,tikz, graphics, xypic, latexsym, tensor}
\usepackage{graphicx,subcaption}
\usepackage[colorlinks=true,linkcolor=blue,citecolor=blue]{hyperref}

\usetikzlibrary{arrows}

\usepackage[scaled]{helvet} 
\usepackage{courier} 
\usepackage[mathbf]{euler}
\usepackage{lhlocal} 

\usepackage{pgfplots}
\usetikzlibrary{arrows.meta}

\numberwithin{equation}{section}

\definecolor{qqwuqq}{rgb}{0,0,0}

\begin{document}

\date{\today}

\title[Rotational Surfaces]{Rotational Surfaces with second fundamental form 
of constant length}

\author[A. Barreto]{Alexandre Paiva Barreto} 
\address{Departamento de Matem\'atica, 
	Universidade Federal de S\~ao Carlos (UFSCar),
	Brasil}
\email{alexandre@dm.ufscar.br}
\thanks{All authors are partially supported by CNPq(Brasil)}

\author[F. Fontenele]{Francisco Fontenele}
\address{Departamento de Geometria, Universidade Federal Fluminense (UFF), 
Brasil.}
\email{fontenele@mat.uff.br}

\author[L. Hartmann]{Luiz Hartmann}
\address{Departamento de Matem\'atica, 
	Universidade Federal de S\~ao Carlos (UFSCar),
	Brasil}
\email{hartmann@dm.ufscar.br}
\urladdr{http://www.dm.ufscar.br/profs/hartmann}


\begin{abstract} 
We obtain an infinite family of complete non embedded rotational surfaces
in $\mathbb R^3$ whose second fundamental forms have length equal to one at any 
point. Also we prove that a complete rotational surface with second 
fundamental form of constant length is either a round sphere, a circular 
cylinder or, up to a homothety and a 
rigid motion, a member of that family. In particular, the 
round sphere and the circular cylinder are
the only complete embedded rotational surfaces in $\mathbb R^3$
with second fundamental form of constant length.
\end{abstract}

\maketitle

\tableofcontents

\section{Introduction}

A surface $S$ in the $3$-dimensional Euclidean space 
is called a Weingarten surface if there exists some relation
\begin{equation}\label{relW}
	W(\lambda_1,\lambda_2)=0, 
\end{equation}
among its principal curvatures $\lambda_1$ and $\lambda_2$. Since the principal 
curvatures of a surface can always be determined from its mean curvature $H$ 
and its Gaussian curvature $K$, and vice-versa, the relation \eqref{relW} can 
always be rewritten as a relation $U(H,K)=0$.

\bigskip

Weingarten surfaces is a classical topic in Differential Geometry that began 
with the works of Weingarten in the middle of the 19th century 
\cite{Wei1,Wei2} and that has been a subject of interest for many authors 
since then (see Chern \cite{Che1}, Hartman and 
Winter \cite{HW}, Hopf \cite{Hop1}, Voss \cite{Vos}, Rosenberg and S\'{a} 
Earp \cite{RSE}, K\"{u}hnel and Steller \cite{KS}, L\'{o}pez 
\cite{Lop1,Lop3,Lop,Lop4}, to name just a few).
\bigskip

Minimal surfaces, surfaces with constant mean curvature and surfaces with 
constant Gaussian curvature are classical examples of Weingarten surfaces. 
Another well known class (generalizing the previous ones) is that of the 
linear 
Weingarten surfaces, \ie Weingarten surfaces verifying either the relation
\begin{equation}\label{relLWI}
	W(\lambda_1,\lambda_2)= a\lambda_{1}+b\lambda_{2}= c
\end{equation}
\noindent or the relation
\begin{equation}\label{relLWII}
	U(H,K)=aH+bK=c,
\end{equation}
where $a,b,c\in\mathbb{R}$ are constants such that $a$ and $b$ do not vanish simultaneously.

\bigskip

The complete classification of Weingarten surfaces is far from being 
achieved. The 
existent results deal mostly with the linear case, 
sometimes making use of additional topological/geometric hypothesis and/or 
working with important subclasses of surfaces such as revolution surfaces \cite{Hop1,Lop,KS,RSE}, tubes along curves and cyclic surfaces \cite{Lop3,Lop4}, ruled surfaces 
and helicoidal surfaces \cite{Kuh}, translation surfaces \cite{DGW,LM}, etc. In general, the approaches used to treat the linear case do not apply to the non-linear case. Therefore, results concerning non-linear Weingarten surfaces are more rare \cite{RSE,KS}. 

\bigskip

In this paper we study rotational surfaces  in the $3$-dimensional Euclidean 
space whose second fundamental forms have constant length (recall that the 
squared length 
$|A|^2$ of the second fundamental form of a surface in $\mathbb R^3$ is 
defined as the trace of $A^2$, where $A$ is its shape operator). In other 
words, we study rotational Weingarten surfaces 
that satisfy the non-linear relation
\begin{equation}\label{relacao segforma cte}%
	W(\lambda_1,\lambda_2)=\left(  \lambda_{1}\right)  ^{2}+\left(  
	\lambda_{2}\right)^{2}=c, 
\end{equation}
or equivalently
\begin{equation}\label{relacao segforma cte2}%
U(H,K)=4H^{2}-2K=c,
\end{equation}
for some $c>0$.

\bigskip

In this case we prove the following result (notice that since the property of 
having constant $|A|$ is invariant by homotheties in $\mathbb R^3$, 
we 
can assume without loss of generality that  $c=1$):

\begin{theorem}\label{Family}
There are two infinite families $\mathcal F_1$ and $\mathcal F_2$ of complete 
non embedded rotational surfaces in $\mathbb R^3$ with $|A|=1$. The family 
$\mathcal F_1$ is one parameter and its members are periodic $C^{\infty}$ 
surfaces, while the members of $\mathcal F_2$ are $C^3$ surfaces. Moreover, any 
complete rotational
$C^2$ surface with $|A|=1$ is either a round sphere of radius $\sqrt{2}$, a
circular cylinder of radius 1 or, up to a rigid motion in $\mathbb R^3$, a
member of one of the two families.
\end{theorem}

\begin{cor}\label{Embedded}
	The only complete embedded rotational $C^2$-surfaces in $\mathbb R^3$ with 
	second
	fundamental form of constant length are the
	round sphere and the circular cylinder.
\end{cor}

Non trivial examples of compact surfaces (embedded or immer\-sed) with second 
fundamental form of constant length are unknown by the authors.
In view of Theorem 1.1 and Corollary 1.2, one can then formulate the following question:


\begin{question}
- Is there a compact surface, other than the round sphere, embedded/immersed in 
the Euclidean 3-space whose second fundamental form has constant length?
\end{question}


It is worth to point out that the above question has a negative answer in the class of the
compact surfaces with positive Gaussian curvature \cite[Theorem 5 on p. 347 and
Section 4]{Ale}  (see also \cite{Kor} or \cite[Theorem 2.3]{FS}).

\bigskip

The importance of the class of hypersurfaces whose second fundamental forms 
have constant length goes beyond the context of Weingarten  surfaces. We 
mention \cite{DX} (see also \cite{Gua}), where it is proved that the 
generalized cylinders are the only complete embedded self-shrinkers in $\R^3$ 
with polynomial volume growth whose second fundamental forms have constant 
length.

\bigskip

The paper is organized as follows. In Section \ref{convexprofcurve}, we prove 
that the profile curve of any rotational surface in $\mathbb R^3$ with 
$|A|\equiv 1$ is convex, \ie its
signed curvature does not change signal. This fact enable us to reduce the 
study of rotational surfaces with $|A|\equiv 1$ to the study of the 
trajectories of a certain vector field in the plane. 
The study of this vector field is made in Section \ref{phasefundvectfield}. 
Finally, we prove Theorem \ref{Family} in Section \ref{proofmaintheorem}.

\subsection*{Acknowledgments} The authors would like to thank Thiago de Melo 
(IGCE-UNESP) for helpful conversations during the preparation of this work and 
for his help with the figures.

\section{Convexity of the profile curves}\label{convexprofcurve}

Our goal in this section is to prove that the profile curve $\mathcal{C}$ of
any rotational $C^2$-surface $M\subset\mathbb R^3$, whose shape operator
$A$ has length $|A|\equiv 1$, is convex. By applying a rigid motion of
$\mathbb R^3$ if necessary, we can assume that $\mathcal{C}$ is contained in
the $xz$-plane and that the axis of revolution is the $x$-axis.

Let $\alpha(t)=(x(t),0,z(t)),\;t\in\mathbb (a,b)$, be a
parametrization of $\mathcal{C}$ such that $||\alpha'(t)||=1$ and $z(t)>0$
for all $t$, and let $\theta:(a,b)\to\mathbb R$ be a continuous
(and, hence, of class $C^1$) function satisfying
\begin{equation}\label{Rot1}
\alpha'(t)=(x'(t),0,z'(t))=(\cos\theta(t),0,\sin\theta(t)),\;\;\;t\in (a,b).
\end{equation}
It is easy to see that the function $\theta$ satisfying
\Eqref{Rot1} is unique up to an integer multiple of $2\pi$. The
principal curvatures of $M$ are given by (see \eg \cite{Lop})
\begin{equation}\label{Rot2}
\lambda_1(t)=\theta'(t),\;\;\;\;\lambda_2(t)=-\frac{\cos\theta(t)}{z(t)}.
\end{equation}
Since $|A|^2=\lambda_1^2+\lambda_2^2\equiv 1$ by hypothesis, one
then has
\begin{equation}\label{Rot3}
\theta'(t)^2+\frac{\cos^2\theta(t)}{z^2(t)}=1,\;\;\;t\in
(a,b).
\end{equation}

As we observed in the introduction, the proof of Theorem
\ref{Family} will be based on a careful study of the trajectories
of a suitable vector field in the plane. The fundamental property
of the profile curves that makes this approach possible is
provided by the following proposition (recall that the signed curvature 
of $\alpha$ is $\theta'$):

\begin{prop}\label{monotona}
	The function $\theta:(a,b)\to\mathbb R$ is monotone.
\end{prop}

Before proceeding to the proof of Proposition \ref{monotona}, let us
explain how to relate profile curves with the trajectories of a specific
vector field.

Let $\alpha$ and $\theta$ be as above. Assuming that 
$\theta$ is monotone,
re\-pa\-ra\-me\-tri\-zing $\alpha$ we can assume that $\theta'\geq 0$. Then,
by \Eqref{Rot1} and \eqref{Rot3},
\begin{equation}\label{sistema}
\begin{cases}
\theta'(t)=\sqrt{1-\frac{\cos^2\theta(t)}{z^2(t)}},\\
z'(t)=\sin\theta(t).
\end{cases}
\end{equation}
Let
$X:\Omega\to\mathbb R^2$ be the (smooth) vector field defined by
\begin{equation}\label{VectorField}
X(\theta,z)=\left(\sqrt{1-\frac{\cos^2\theta}{z^2}},\sin\theta\right),
\end{equation}
where $\Omega=\{(\theta,z)\in\mathbb R^2:z>|\cos\theta|\}$.
As long as $z(t)>|\cos\theta(t)|$, the system in \Eqref{sistema} can
be rewritten as
\begin{equation}\label{Trajetoria}
(\theta'(t),z'(t))=X(\theta(t),z(t)),
\end{equation}
and so the curve $t\mapsto(\theta(t),z(t))$ is a trajectory of
$X$. A representation of $\Omega$ and the vector field $X$ can be seen in
Figure \ref{Figura1}.
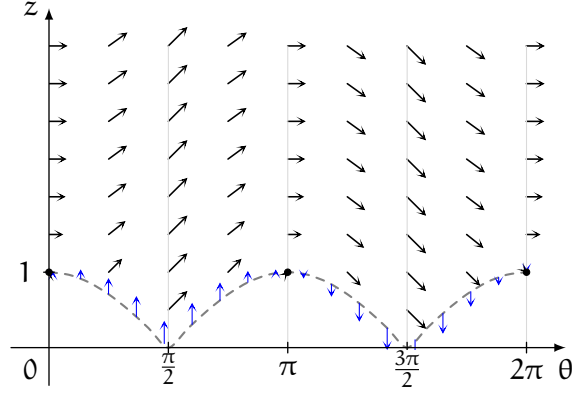
\begin{figure}[htp!]
	\begin{tikzpicture}%
	[line cap=round,>=latex,x=1cm,y=1cm]
	\clip(-.5,-1.) rectangle (7,4.6);
	\begin{scope}[domain=0:2*pi]
	\draw[dashed,black!50,thick] plot[smooth] (\x,{abs(cos(\x r))});
	\end{scope}
	\begin{axis}[
	anchor=origin,
	x=1cm,y=1cm,
	domain=0.05:2*pi,
	axis lines=none,
	axis line style={},
	xtick=\empty,
	ytick=\empty,
	]
	\addplot[%
	blue,
	quiver={u=0,v=0.3 * sin (\x r)},
	-{Stealth[length=3.2pt,width=3.2pt,inset=2pt]},
	line width=.4pt,
	samples=18,
	] {abs(cos(\x r))};
	\end{axis}
	\begin{scope}[line width=.6pt,-{
		Stealth[length=3.2pt,width=3.2pt,inset=2pt]%
	}]
	\clip(0,0) rectangle (6.6,4.6);
	\draw [color=qqwuqq] (0.,1.)-- (0.,1.);
\draw [color=qqwuqq] (0.,1.5)-- (0.18633899812498247,1.5);
\draw [color=qqwuqq] (0.,2.)-- (0.21650635094610965,2.);
\draw [color=qqwuqq] (0.,2.5)-- (0.229128784747792,2.5);
\draw [color=qqwuqq] (0.,3.)-- (0.23570226039551584,3.);
\draw [color=qqwuqq] (0.,3.5)-- (0.23957871187497748,3.5);
\draw [color=qqwuqq] (0.,4.)-- (0.24206145913796356,4.);
\draw [color=qqwuqq] (0.785,1.)-- (0.9617062952763415,1.1767062952763414);
\draw [color=qqwuqq] (0.785,1.5)-- (1.0054541925614173,1.6767062952763414);
\draw [color=qqwuqq] (0.785,2.)-- (1.0188402845909412,2.1767062952763414);
\draw [color=qqwuqq] (0.785,2.5)-- (1.0247832737420322,2.6767062952763414);
\draw [color=qqwuqq] (0.785,3.)-- (1.0279506385142394,3.1767062952763414);
\draw [color=qqwuqq] (0.785,3.5)-- (1.0298406586867184,3.6767062952763414);
\draw [color=qqwuqq] (0.785,4.)-- (1.031059585617779,4.176706295276341);
\draw [color=qqwuqq] (1.57,0.5)-- (1.8199996829316838,0.7499999207329586);
\draw [color=qqwuqq] (1.57,1.)-- (1.8199999207329587,1.2499999207329586);
\draw [color=qqwuqq] (1.57,1.5)-- (1.819999964770207,1.7499999207329586);
\draw [color=qqwuqq] (1.57,2.)-- (1.819999980183242,2.2499999207329586);
\draw [color=qqwuqq] (1.57,2.5)-- (1.8199999873172752,2.7499999207329586);
\draw [color=qqwuqq] (1.57,3.)-- (1.8199999911925522,3.2499999207329586);
\draw [color=qqwuqq] (1.57,3.5)-- (1.8199999935292221,3.7499999207329586);
\draw [color=qqwuqq] (1.57,4.)-- (1.8199999950458108,4.249999920732959);
\draw [color=qqwuqq] (2.355,1.)-- (2.531987727162108,1.176987727162108);
\draw [color=qqwuqq] (2.355,1.5)-- (2.5755545088241694,1.676987727162108);
\draw [color=qqwuqq] (2.355,2.)-- (2.5888934883477996,2.176987727162108);
\draw [color=qqwuqq] (2.355,2.5)-- (2.594816481690816,2.676987727162108);
\draw [color=qqwuqq] (2.355,3.)-- (2.5979733994507788,3.176987727162108);
\draw [color=qqwuqq] (2.355,3.5)-- (2.5998572521363332,3.676987727162108);
\draw [color=qqwuqq] (2.355,4.)-- (2.6010722271465747,4.176987727162108);
\draw [color=qqwuqq] (3.14,1.)-- (3.1403981632291305,1.0003981632291217);
\draw [color=qqwuqq] (3.14,1.5)-- (3.3263391871876626,1.5003981632291217);
\draw [color=qqwuqq] (3.14,2.)-- (3.3565064424757134,2.000398163229122);
\draw [color=qqwuqq] (3.14,2.5)-- (3.3691288400996986,2.500398163229122);
\draw [color=qqwuqq] (3.14,3.)-- (3.375702297762325,3.000398163229122);
\draw [color=qqwuqq] (3.14,3.5)-- (3.3795787388839442,3.500398163229122);
\draw [color=qqwuqq] (3.14,4.)-- (3.3820614796046087,4.0003981632291215);
\draw [color=qqwuqq] (3.925,1.)-- (4.101424415167119,0.8235755848328808);
\draw [color=qqwuqq] (3.925,1.5)-- (4.14535383088424,1.3235755848328807);
\draw [color=qqwuqq] (3.925,2.)-- (4.158787068861315,1.8235755848328807);
\draw [color=qqwuqq] (3.925,2.5)-- (4.164750061277844,2.3235755848328807);
\draw [color=qqwuqq] (3.925,3.)-- (4.167927875502874,2.8235755848328807);
\draw [color=qqwuqq] (3.925,3.5)-- (4.169824064154543,3.3235755848328807);
\draw [color=qqwuqq] (3.925,4.)-- (4.1710469434715485,3.8235755848328807);
\draw [color=qqwuqq] (4.71,0.5)-- (4.959997146375503,0.2500007134030705);
\draw [color=qqwuqq] (4.71,1.)-- (4.959999286596929,0.7500007134030705);
\draw [color=qqwuqq] (4.71,1.5)-- (4.95999968293222,1.2500007134030704);
\draw [color=qqwuqq] (4.71,2.)-- (4.959999821649423,1.7500007134030704);
\draw [color=qqwuqq] (4.71,2.5)-- (4.9599998858556456,2.2500007134030704);
\draw [color=qqwuqq] (4.71,3.)-- (4.959999920733092,2.7500007134030704);
\draw [color=qqwuqq] (4.71,3.5)-- (4.959999941763091,3.2500007134030704);
\draw [color=qqwuqq] (4.71,4.)-- (4.959999955412368,3.7500007134030704);
\draw [color=qqwuqq] (5.495,1.)-- (5.672268710110554,0.8227312898894459);
\draw [color=qqwuqq] (5.495,1.5)-- (5.715654778717304,1.3227312898894459);
\draw [color=qqwuqq] (5.495,2.)-- (5.728946679600427,1.8227312898894459);
\draw [color=qqwuqq] (5.495,2.5)-- (5.734849684789206,2.322731289889446);
\draw [color=qqwuqq] (5.495,3.)-- (5.737996158082171,2.822731289889446);
\draw [color=qqwuqq] (5.495,3.5)-- (5.739873844335265,3.322731289889446);
\draw [color=qqwuqq] (5.495,4.)-- (5.741084867929778,3.822731289889446);
\draw [color=qqwuqq] (6.28,1.)-- (6.280796325448285,0.9992036745517155);
\draw [color=qqwuqq] (6.28,1.5)-- (6.466339754372633,1.4992036745517154);
\draw [color=qqwuqq] (6.28,2.)-- (6.496506717063363,1.9992036745517154);
\draw [color=qqwuqq] (6.28,2.5)-- (6.509129006154776,2.4992036745517154);
\draw [color=qqwuqq] (6.28,3.)-- (6.515702409862338,2.9992036745517154);
\draw [color=qqwuqq] (6.28,3.5)-- (6.519578819910551,3.4992036745517154);
\draw [color=qqwuqq] (6.28,4.)-- (6.522061541004326,3.9992036745517154);
	\end{scope}
	\draw[black!15] (1.57075,0) -- (1.57075,4) (3.1415,0) -- (3.1415,4) 
	(4.7125,0) -- (4.7125,4) (6.283,0) -- (6.283,4);
	%
	\draw[fill] (0,1) circle (1.2pt) (3.1415,1) circle (1.2pt) (6.2831,1) 
	circle (1.2pt) ;
	\begin{scope}[line width=.5pt]
	\draw[->] (-.5,0) -- (6.8,0)node[below]{\small$\theta$};
	\draw[->] (0,-.5) -- (0,4.5)node[left]{\small$z$};
	\foreach \x/\l in 	
{1.57075/{\frac{\pi}{2}},3.1415/{\pi},4.71225/{\frac{3\pi}{2}},6.283/{2\pi}}{%
		\draw[shift={(\x,0)}] (0pt,2pt) -- (0pt,-2pt);
		\draw[shift={(\x,0)}] (0,-8pt)node[fill=white,inner sep=0pt] 
		{\small$\l$};
	}
	\foreach \y in {1}{
		\draw[shift={(0,\y)}] (2pt,0pt) -- (-2pt,0pt) node[left] {\small 
			$\y$};
	}
	\draw (0,0) node[below left] {\small$0$};
	\end{scope}
	\end{tikzpicture}
	\caption{Graphic representation of $\Omega$ and the vector field 
	$X$ for $\theta\in[0,2\pi]$.}\label{Figura1}
\end{figure}

Conversely, given a trajectory
$\varphi(t)=(\theta(t),z(t)),\;t\in(a,b),$ of $X$ and
$t_0\in(a,b)$, consider the curve
$\alpha(t)=(x(t),0,z(t)),\;t\in(a,b),$ where
$$
x(t)=\int_{t_0}^t\cos\theta(s)ds.
$$
Using \Eqref{Rot2}, \eqref{VectorField} and \eqref{Trajetoria} one
easily proves that the surface in $\mathbb R^3$ obtained by the
rotation of the image of $\alpha$ around the $x$-axis satisfies
$|A|\equiv 1$.

In the proof of Proposition \ref{monotona}, as well as in the proofs of
later results, we will use the following technical lemma. In its statement,
$\alpha(t)=(x(t),0,z(t))$ and $\theta(t)$ are as in the beginning
of this section.

\begin{lemma}\label{ThetaLinhaZero}
	For any $t\in(a,b)$, the following assertions hold:

	\begin{enumerate}
		\item[(i)] $\theta'(t)=0$ if, and only if, $z(t)=1$ and
		$\sin\theta(t)=0$.

		\item[(ii)] If $z(t)<1$ then $|z'(t)|\geq|\theta'(t)|$.

		\item[(iii)] If $\theta'(t)=0$, then there exists $\delta>0$
		such that $z(s)\geq 1$, $s\in(t-\delta,t+\delta)$.
	\end{enumerate}
\end{lemma}

\begin{proof} (i) If $\theta'(t)=0$ then, by \Eqref{Rot3}, the function
	$\cos^2\theta/z^2$ attains a maximum at $t$. Hence,
	\begin{eqnarray}
	0&=&\left[\frac{\cos^2\theta}{z^2}\right]'(t)
	=\left[\frac{-2\theta'\cos\theta\sin\theta z^2-2\cos^2\theta 
		zz'}{z^4}\right](t)\nonumber\\
	&=&\frac{-2\cos^2\theta(t)z'(t)}{z^3(t)}.\nonumber
	\end{eqnarray}
	Since $\cos\theta(t)\neq 0$, one obtains from the above equality and 
	\Eqref{Rot1} that
	$$
	\sin\theta(t)=z'(t)=0.
	$$
	Using this information in \Eqref{Rot3}, one concludes that
	$z(t)=1$. The converse is an immediate consequence of
	\Eqref{Rot3}.
	
	\vskip5pt
	
	\noindent(ii) From \Eqref{Rot3} and $z(t)<1$ one obtains
	$$
	1=\theta'(t)^2+\frac{\cos^2\theta(t)}{z^2(t)}\geq\theta'(t)^2+\cos^2\theta(t),
	$$
	and so
	$$
	\theta'(t)^2\leq 1-\cos^2\theta(t)=\sin^2\theta(t)=z'(t)^2.
	$$
	The conclusion now follows by taking square roots in the above inequality.
	
	\vskip5pt
	
	\noindent(iii) Supposing, by contradiction, that the conclusion
	does not hold, we have $z(t_n)<1$ for some sequence $(t_n)$ that
	converges to $t$. Since, by (i), $t_n\neq t$ for all $n$,
	passing to a subsequence and reparametrizing $\alpha$ if necessary, one can assume that $t_n<t$, for all
	$n$.
	
	We claim that
	\begin{equation}\label{Rot3bb}
	z(s)<1,\;\;\;s\in(a,t).
	\end{equation}
	Indeed, if $z(c)\geq 1$ for some $c\in(a,t)$ then, since $t_n\to
	t$ and $z(t_n)<1$ for all $n$, there is $d\in(c,t)$ such that
	\begin{equation}\label{Rot3bc}
	z(d)=\inf\{z(s):c\leq s\leq t\}<1.
	\end{equation}
	Hence, $z'(d)=0$. On the other hand, from (i), (ii) and \Eqref{Rot3bc} 
	one 
	obtains $|z'(d)|\geq|\theta'(d)|>0$.
	This contradiction proves \Eqref{Rot3bb}.
	
	By (i), (ii) and \Eqref{Rot3bb},
	\begin{equation}\label{Rot3bbb}
	\theta'(s)\neq 0\;\;\text{and}\;\;z'(s)\neq 0,\;\;\;s\in(a,t).
	\end{equation}
	Then, by \Eqref{Rot3bb} and $z(t)=1$,
	\begin{equation}\label{Rot3c}
	\sin\theta(s)=z'(s)>0,\;\;\;s\in(a,t).
	\end{equation}
	Since $\sin\theta(t)=0$ by (i), we have two possibilities:
	
	\vskip5pt
	
	\noindent{a)} $\theta(t)=2k\pi$, for some $k\in\mathbb Z$.
	
	\vskip5pt
	
	\noindent{b)} $\theta(t)=(2k+1)\pi$, for some $k\in\mathbb Z$.
	
	\vskip5pt
	
	Assuming a), from \Eqref{Rot3c} one obtains
	$$
	\theta(t)+\pi>\theta(s)>\theta(t),\;\;\;s\in(a,t).
	$$
	Then, by the first inequality of \Eqref{Rot3bbb},
	\begin{equation}\label{Rot3d}
	\theta'(s)<0,\;\;\;s\in(a,t).
	\end{equation}	
	From \Eqref{Rot3bb}, \eqref{Rot3c}, \eqref{Rot3d} and (ii), we obtain
	$$
	z'(s)\geq -\theta'(s),\;\;\;s\in(a,t).
	$$
	Hence, for fixed $s_1\in(a,t)$, we have 
	$$
	1=z(t)=z(s_1)+\int_{s_1}^{t}z'(s)ds\geq
	z(s_1)-\int_{s_1}^{t}\theta'(s)ds=z(s_1)+\theta(s_1)-\theta(t),
	$$
	and so
	$$
	1-z(s_1)\geq\theta(s_1)-\theta(t).
	$$
	It now follows from \Eqref{Rot3d} and the fact that the cosine function 
	is decreasing on $[0,\pi]$, that
	\begin{equation}\label{Rot3f}
	\begin{aligned}
	\cos(1-z(s_1)) \leq& \ \cos(\theta(s_1)-\theta(t))\\
	=& \ \cos\theta(s_1)\cos\theta(t)+\sin\theta(s_1)\sin\theta(t)\\
	=& \ \cos\theta(s_1).
	\end{aligned}
	\end{equation}
	Since, by \Eqref{Rot3} and \eqref{Rot3d},
	$z(s_1)>|\cos\theta(s_1)|$, inequalities \Eqref{Rot3bb} and
	\eqref{Rot3f} imply
	\begin{equation}\label{Rot3g}
	\cos(1-z(s_1))<z(s_1)<1,
	\end{equation}
	contradicting the fact, easily verified, that $\cos(1-x)>x$, 
	for all $x\in[0,1)$.
	
	A reasoning entirely similar to the above shows that b) cannot
	occur either. Hence, $z(s)\geq 1$ on a
	neighbourhood of $t$.
\end{proof}

\noindent{\bf Proof of Proposition \ref{monotona}:} Suppose, by
contradiction, that $\theta$ is not monotone. Then there exists
$t_1<t_2<t_3$ in $(a,b)$ such that either i) or ii) below holds:

\vskip5pt

\noindent i) $\theta(t_1)<\theta(t_2)$ and
$\theta(t_2)>\theta(t_3)$.

\vskip5pt

\noindent ii) $\theta(t_1)>\theta(t_2)$ and
$\theta(t_2)<\theta(t_3)$.

\vskip5pt

Assuming i), we have
$$
\lambda:=\sup\{\theta(t):t\in[t_1,t_3]\}>\max\{\theta(t_1),\theta(t_3)\}.
$$
Define
$$
\xi:=\inf\{t\geq t_1:\theta(t)=\lambda\},\;\;\;\eta:=\sup\{t\leq 
t_3:\theta(t)=\lambda\}.
$$
Since $\theta$ attains a local maximum at $\xi$ and at $\eta$, we
have $\theta'(\xi)=\theta'(\eta)=0$. Then, by Lemma
\ref{ThetaLinhaZero} (i), $z(\xi)=z(\eta)=1$ and
$\sin\theta(\xi)=\sin\theta(\eta)=0$. The latter implies that
either $\theta(\xi)=\theta(\eta)=2k\pi$ or
$\theta(\xi)=\theta(\eta)=(2k+1)\pi$, for some $k\in\mathbb Z$. If
$\theta(\xi)=\theta(\eta)=2k\pi$, one has  $\mu\in (\eta,t_3]$, such that
$$
z'(t)=\sin\theta(t)<0,\;\;\;t\in(\eta,\mu).
$$
Then, $z(t)<z(\eta)=1$, for all $t\in(\eta,\mu)$, contradicting
Lemma \ref{ThetaLinhaZero} (iii).

If $\theta(\xi)=\theta(\eta)=(2k+1)\pi$, there exists $\nu\in [t_1,\xi)$
such that
$$
z'(t)=\sin\theta(t)>0,\;\;\;t\in(\nu,\xi).
$$
Then, $z(t)<z(\xi)=1$, for all  $t\in(\nu,\xi)$, which also
contradicts Lemma \ref{ThetaLinhaZero} (iii).

A reasoning entirely similar to the above shows that ii) can not
occur either. Hence, the function $\theta$ is monotone.\qed

\section{Phase portrait of the fundamental vector field}\label{phasefundvectfield}

With the aim to prove Theorem \ref{Family}, we study in this section the trajectories of the
vector field $X$ defined by \Eqref{VectorField}. This study will
be carried out through a series of technical lemmas. 

Since the trajectories
of $X$ are invariant by horizontal translations by multiples of
$2\pi$ (that is, if $\varphi(t)=(\theta(t),z(t))$ is a trajectory
of $X$ then so is the curve $\psi(t)=(\theta(t)+2n\pi,z(t))$ for
any $n\in\mathbb Z$), it is sufficient to consider the
trajectories that pass through some point
$(\theta_0,z_0)\in\Omega$ such that $0\leq\theta_0\leq 2\pi$.

\begin{lemma}\label{SuperaBarreira}
	Let $(\theta_0,z_0)$ be a point in $\Omega$ such that
	$0<\theta_0<\pi$ and $z_0\leq 1$. If
	$\varphi(t)=(\theta(t),z(t)),\;t\in(a,b),$ is the maximal integral
	curve of $X$ satisfying $\varphi(0)=(\theta_0,z_0)$, then there
	exists $c\in(0,b)$ such that $z(c)>1$.
\end{lemma}

\begin{proof}
	We can assume that $z_0<1$, for otherwise the conclusion follows
	immediately from $z'(0)=\sin\theta(0)>0$. Suppose, by
	contradiction, that the conclusion does not hold. Then, by the
	definition of $\Omega$,
	\begin{equation}\label{Rot12}
	\theta_0<\theta(t)<\pi,\;\;\;t\in(0,b),
	\end{equation}
	and so
	\begin{equation}\label{Rot13}
	z'(t)=\sin\theta(t)>0\;\;\text{and}\;\;z(t)<1,\;\;\;t\in(0,b).
	\end{equation}
	Let
	\begin{equation}\label{Rot15}
	\theta_+=\lim_{t\to
		b}\theta(t)\;\;\;\text{and}\;\;\;z_+=\lim_{t\to b}z(t).
	\end{equation}
	By \Eqref{Rot12} and \eqref{Rot13}, $\frac{\pi}{2}<\theta_+\leq\pi$ and $z_+\leq
	1$. From the maximality of $\varphi$ and the fact that $X$ has no
	singularities in $\Omega$, one obtains
	$(\theta_+,z_+)\in\partial\Omega$, and so
	\begin{equation}\label{Rot16}
	z_+=|\cos\theta_+|=-\cos\theta_+.
	\end{equation}
	
	We have two cases to consider:
	
	\vskip5pt
	
	\noindent i) $\theta_+<\pi$ (and so $z_+<1$).
	
	\vskip5pt
	
	\noindent ii) $\theta_+=\pi$ (and so $z_+=1$).
	
	\vskip5pt
	
	Since the vectors of $X$ on the boundary of  $\Omega$ points inward, we 
	can use transversality to conclude that case i) can not occur.  However, 
	we will discard this case by a direct argument.
	 Consider the (positive) function
	$\xi:(0,b)\to\mathbb R$ defined by $\xi(t)=z(t)+\cos\theta(t)$. By
	\Eqref{Rot15} and \eqref{Rot16},
	\begin{equation}\label{Rot17}
	\lim_{t\to b}\xi(t)=z_++\cos\theta_+=0.
	\end{equation}
	Using now \Eqref{sistema}, \eqref{Rot15} and \eqref{Rot16}, one
	obtains
	\begin{equation*}\label{Rot18}
	\begin{aligned}
	\lim_{t\to b}\xi'(t)=&\ \lim_{t\to 
		b} \ \Bl \sin\theta(t)-\theta'(t)\sin\theta(t) \Br\\
	=&\ \sin\theta_+-\sin\theta_+\sqrt{1-\frac{\cos^2\theta_+}{z_+^2}}\\
	=&\ \sin\theta_+>0.
	\end{aligned}
	\end{equation*}
	Then, there is $t_0\in(0,b)$
	such that $\xi'(t)>\sin\theta_+/2,\;t\in[t_0,b),$ and so
	\begin{equation*}\label{Rot20}
	\xi(t)-\xi(t_0)=\int_{t_0}^t\xi'(s)ds>\frac{\sin\theta_+}{2}(t-t_0),\;\;\;
	t>t_0.
	\end{equation*}
	Letting $t\to b$ in the above inequality, and using \Eqref{Rot17},
	one obtains
	$$
	-\xi(t_0)\geq\frac{\sin\theta_+}{2}(b-t_0)>0,
	$$
	contradicting the fact that $\xi(t)>0$, for all $t$.
	
	\vskip5pt
	
	Suppose now ii). From \Eqref{Rot13} and Lemma \ref{ThetaLinhaZero}
	(ii), we obtain $z'(t)\geq\theta'(t)$ for all $t\in(0,b)$, and so
	\begin{equation*}\label{Rot21}
	1>z(t)=z(0)+\int_0^tz'(s)ds \geq 
	z(0)+\int_0^t\theta'(s)ds=z_0+\theta(t)-\theta_0,
	\end{equation*}
	for every $t\in(0,b)$. Taking the limit when $t\to b$ in the above
	inequality, and using \Eqref{Rot15} and b), we obtain
	\begin{equation}\label{Rot21a}
	1\geq z_0+\theta_+-\theta_0=z_0+\pi-\theta_0>\pi-\theta_0,
	\end{equation}
	and thus $0<\pi-\theta_0<1$. Choosing $k\in\mathbb N$ such that
	\begin{equation}\label{Rot23}
	\frac{1}{k+1}\leq\pi-\theta_0<\frac{1}{k},
	\end{equation}
	one has, since the cosine function is decreasing on $(0,\pi)$,
	\[
	0<\cos\left(\frac{1}{k}\right)<\cos(\pi-\theta_0)=-\cos\theta_0\leq|\cos\theta_0|.
	\]
	Then, by \Eqref{Rot3} and the above inequality,
	\begin{equation*}
	\frac{\cos^2(1/k)}{z_0^2}<\frac{\cos^2\theta(0)}{z^2(0)}\leq 1,
	\end{equation*}
	and so $z_0>\cos(1/k)$. Using now that $\cos x>1/(1+x)$ for every
	$x\in(0,1]$, one concludes that $z_0>k/(k+1)$. Hence, by
	\Eqref{Rot21a} and \eqref{Rot23},
	\begin{equation*}
	1\geq z_0+\pi-\theta_0>\frac{k}{k+1}+\frac{1}{k+1}=1,
	\end{equation*}
	which is obviously false. This contradiction finishes the proof of
	the lem\-ma.
\end{proof}

\vskip10pt

\begin{lemma}\label{Alcance}
	Let $(\theta_0,z_0)\in\Omega$ such that $0\leq\theta_0<\pi$. If
	$\varphi(t)=(\theta(t),z(t)),\;t\in(a,b),$ is the maximal integral
	curve of $X$ satisfying $\varphi(0)=(\theta_0,z_0)$, then there
	exists $t_0\in(0,b)$ such that $\theta(t_0)=\pi$.
\end{lemma}

\begin{proof}
	Assuming, by contradiction, that the conclusion
	does not hold, one has $0<\theta(t)<\pi,\;t\in(0,b)$, and so
	\begin{equation}\label{Rot29}
	z'(t)=\sin\theta(t)>0,\;\;t\in(0,b).
	\end{equation}
	Let $c\in[0,b)$ such that $z(c)>1$ (such a number $c$ exists by
	Lemma \ref{SuperaBarreira}). Since $\theta$ is bounded above and, by \Eqref{sistema} and
	\eqref{Rot29},
	\begin{equation*}
	\theta'(t)\geq\sqrt{1-\frac{1}{z^2(t)}}>\sqrt{1-\frac{1}{z^2(c)}}>0,\;\;\;t\in(c,b),
	\end{equation*} 
	one concludes that $b<\infty$. Then, since $z'(t)\leq 1$, one also has that $z$ is bounded.
	Hence, $\varphi(t)=(\theta(t),z(t))$ converges to a point in $\Omega$ when $t\to b$, but this can not occur because $b<\infty$ (see, for instance, \cite[p.
	91]{Per}). 
\end{proof}

\begin{lemma}\label{Simetrico} Given $z_0>1$ and $n\in\mathbb Z$,
	let $\varphi(t)=(\theta(t),z(t)),\;t\in(a,b),$ be the maximal
	integral curve of $X$ satisfying $\varphi(0)=(n\pi,z_0)$. Then,
	$a=-b$ and $\varphi(-t)=R(\varphi(t))$ for every $t\in(-b,b)$,
	where $R$ denotes the reflection in $\mathbb R^2$ with respect to
	the line $\theta=n\pi$. In short, $\varphi$ is symmetric with
	respect to the line $\theta=n\pi$.
\end{lemma}

\begin{proof}
	Consider the curve $\sigma:(-b,-a)\to\Omega$ defined by
	\begin{equation}\label{Rot24}
	\sigma(t)=R(\varphi(-t))=(2n\pi-\theta(-t),z(-t)).
	\end{equation}
	It is easy to see that $\sigma$ is an integral curve of $X$. Since
	$\sigma(0)=(n\pi,z_0)=\varphi(0)$, it follows from the maximality
	of $\varphi$ that $a=-b$ and
	$$
	\varphi(t)=\sigma(t)=R(\varphi(-t)),\;\;\;t\in(-b,b).
	$$
\end{proof}

Lemmas \ref{Alcance} and \ref{Simetrico} tell us that to
obtain a picture of the phase portrait of $X$ it is sufficient to
consider the family of trajectories
$\{\varphi_{\lambda}\}_{\lambda>1}$, where
$\varphi_{\lambda}:(-b_{\lambda},b_{\lambda})\to\Omega$ is the
maximal integral curve of $X$ such that $\varphi_
{\lambda}(0)=(\pi,\lambda)$.

From Lemma \ref{Alcance} and the fact that $z'(t)= \sin\theta(t)$ is positive 
on $\theta^{-1}((0,\pi)) $, one concludes, for each $\lambda>1$, that the 
trajectory
$\varphi_{\lambda}:(-b_{\lambda},b_{\lambda})\to\Omega$ either
crosses the ray $\{(\theta,z)\in\mathbb
R^2:\theta=0\;\text{and}\;z>1\}$ or converges to a point
$p_{\lambda}\in\partial\Omega$ when $t\to -b_{\lambda}$. Moreover,
each point $(\theta,z)\in\partial\Omega$ such that
$0\leq\theta<\pi$ is the limit point $p_{\lambda}$ of some
$\varphi_{\lambda}$. The later is clear when $\theta\not = 0$ and $\theta \not = \frac{\pi}{2}$,
as the vector field $X$ can be continuously extended, without singularities, 
to a neighbourhood of $(\theta,z)$, and follows easily for the other two 
values of $\theta$ by a continuity argument. 

The following lemma shows that two distinct
trajectories of the family $\{\varphi_{\lambda}\}_{\lambda>1}$ can
not converge to the same point in $\partial\Omega$. Note that this
fact does not follow from the standard theory of ordinary
differential equations, because the vector field $X$ does not admit a 
differentiable extension to a neighbourhood of any given point in
$\partial\Omega$.

\begin{lemma}\label{UnicLimPoint}
	With the same notation as above, assume for some
	$(\theta_0,z_0)\in\partial\Omega$ that
	\begin{equation}\label{Rot33A}
	\lim_{t\to -b_{\lambda_1}}\varphi_{\lambda_1}(t)=(\theta_0,z_0)=\lim_{t\to
		-b_{\lambda_2}}\varphi_{\lambda_2}(t).
	\end{equation}
	Then $\lambda_1=\lambda_2$ (and hence 
	$\varphi_{\lambda_1}=\varphi_{\lambda_2}$).
\end{lemma}

\begin{proof}
	Assume, by contradiction, that $\lambda_1\neq\lambda_2$, say
	$\lambda_1<\lambda_2$. Setting
	$\varphi_{\lambda_1}=(\theta_1,z_1)$ and
	$\varphi_{\lambda_2}=(\theta_2,z_2)$, from \Eqref{Rot33A} one
	obtains that $\theta_1=\theta_1|_{(-b_{\lambda_1},0]}$ (respectively,
	$\theta_2=\theta_2|_{(-b_{\lambda_2},0]}$) is a diffeomorphism
	from $(-b_{\lambda_1},0]$ (respectively, $(-b_{\lambda_2},0]$) to
	$(\theta_0,\pi]$. Let
	$\psi=\theta_2^{-1}\circ\theta_1:(-b_{\lambda_1},0]\to(-b_{\lambda_2},0]$.
	By the Chain Rule and the Inverse Function Theorem,
	\begin{equation}\label{Rot33B}
	\psi'(t)=(\theta_2^{-1})'(\theta_1(t))\theta_1'(t)=
	\frac{\theta_1'(t)}{\theta_2'(\psi(t))}>0,\;\;\;t\in(-b_{\lambda_1},0].
	\end{equation}
	Since $\lambda_1<\lambda_2$ and $\theta_2(\psi(t))=\theta_1(t)$
	for $t\in(-b_{\lambda_1},0]$, we have $z_2(\psi(t)) > z_1(t)$ and so
	\[
	\theta_2'(\psi(t))=\sqrt{1-\frac{\cos^2\theta_2(\psi(t))}{z_2^2(\psi(t))}}
	\geq \sqrt{1-\frac{\cos^2\theta_1(t)}{z_1^2(t)}}=\theta_1'(t),
	\]
	for all $t\in (-b_{\lambda_1},0]$. Using this inequality in 
	\Eqref{Rot33B}, we obtain
	\begin{equation}\label{Rot33C}
	\psi'(t)\leq 1,\;\;\;t\in(-b_{\lambda_1},0].
	\end{equation}
	Using again the equality $\theta_2(\psi(t))=\theta_1(t)$, it follows from
	the 
	Chan\-ge of 
	Variables Formula that
	\begin{equation*}\label{Rot33D}
	\begin{aligned}
	z_2(0)-z_2(\psi(t))=& \ \int_{\psi(t)}^0z_2'(s)ds=\int_t^0z_2'(\psi(u))
	\psi'(u)du\\
	=& \ \int_t^0\sin\theta_2(\psi(u))\psi'(u)du
	=\int_t^0\sin\theta_1(u)\psi'(u)du,
	\end{aligned}
	\end{equation*}
	for every $t\in (-b_{\lambda},0]$.
	Hence, by \Eqref{Rot33C},
	\begin{equation*}\label{Rot33E}
	\begin{aligned}
	z_2(0)-z_2(\psi(t))=& \ \int_t^0z_1'(u)\psi'(u)du\\
	\leq& \ \int_t^0z_1'(u)du =z_1(0)-z_1(t), \;\;\; 
	t\in(-b_{\lambda_1},0].
	\end{aligned}
	\end{equation*}
	Taking the limit when $t\to
	-b_{\lambda_1}$, and using \Eqref{Rot33A}, one obtains
	$\lambda_2=z_2(0)\leq z_1(0)=\lambda_1$, contradicting our
	assumption $\lambda_1<\lambda_2$. Hence $\lambda_1=\lambda_2$.
\end{proof}

\begin{figure}[htp!]
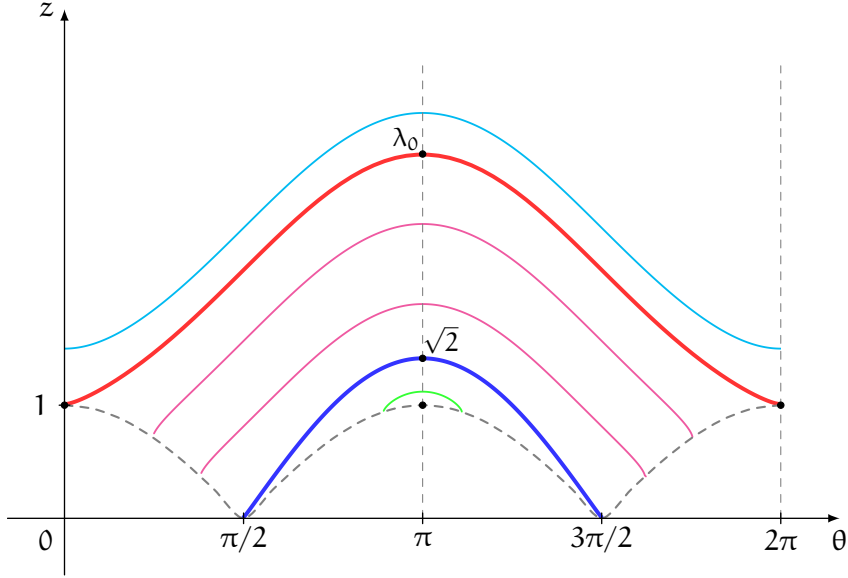

	\begin{tikzpicture}%
	[line cap=round,>=latex,x=1.5cm,y=1.5cm]
	\clip(-.5,-1.) rectangle (7,4.6);
	%
	%
	\begin{scope}[line width=.7pt,cyan!80]
	\clip(0,0) rectangle (2 * pi,4.6);
	\input{curva1}
	\end{scope}
	\begin{scope}[line width=.7pt,magenta!80]
	\clip(0,0) rectangle (2 * pi,4.6);
	\input{curva3}
	\input{curva4}
	\end{scope}
	\begin{scope}[line width=.7pt,green!80]
	\clip(0,0) rectangle (2 * pi,4.6);
	\input{curva5}
	\end{scope}

	\begin{scope}[line width=1.5pt,red!80]
	\clip(0,0) rectangle (2 * pi,4.6);
	\input{curva2}
	\end{scope}
	\begin{scope}[line width=1.5pt,blue!80]
	\clip(0,0) rectangle (2 * pi,4.6);
	\input{sphere}
	\end{scope}
	\begin{scope}[domain=0:2*pi]
	\draw[dashed,black!50,thick] plot[smooth] (\x,{abs(cos(\x r))});
	\end{scope}
	%
	\draw (3,3.25) node[fill=white,above,inner sep=0pt] {\footnotesize 
		$\lambda_0$};
	\draw (3.3,1.45) node[fill=white,above,inner sep=0pt] {\footnotesize 
		$\sqrt{2}$};
	\draw[dashed,black!50] (3.1415,0) -- (3.1415,4) (6.283,0) -- (6.283,4);
	%
	\draw[fill] (0,1) circle (1.2pt) (3.1415,1) circle (1.2pt) (6.2831,1) 
	circle (1.2pt)  (3.1415,3.22) circle (1.2pt) (3.1415,1.4142) circle 
	(1.2pt);
	\begin{scope}[line width=.5pt]
	\draw[->] (-.5,0) -- (6.8,0)node[below]{\small$\theta$};
	\draw[->] (0,-.5) -- (0,4.5)node[left]{\small$z$};
	\foreach \x/\l in 
	{1.57075/{\pi/2},3.1415/{\pi},4.71225/{3\pi/2},6.283/{2\pi}}{%
		\draw[shift={(\x,0)}] (0pt,2pt) -- (0pt,-2pt);
		\draw[shift={(\x,0)}] (0,-8pt)node[fill=white,inner sep=0pt] 
		{\small$\l$};
	}
	\foreach \y in {1}{
		\draw[shift={(0,\y)}] (2pt,0pt) -- (-2pt,0pt) node[left] {\small 
			$\y$};
	}
	\draw (0,0) node[below left] {\small$0$};
	\end{scope}
	\end{tikzpicture}
	\caption{Graphic representation of the phase portrait of the vector field 
	$X$ for 
	$\theta\in[0,2\pi]$.}\label{Figure2}
\end{figure}

\section{Proof of Theorem \ref{Family}}\label{proofmaintheorem}

As before, for each $\lambda>1$ denote by
$\varphi_{\lambda}(t)=(\theta(t),z(t)),\;t\in(-b_{\lambda},b_{\lambda}),$
the maximal integral curve of $X$ such that
$\varphi_{\lambda}(0)=(\pi,\lambda)$. From Lemma
\ref{UnicLimPoint} and the discussion that precedes its statement
one concludes that there is a unique $\lambda_0>1$ such that
$\varphi_{\lambda_0}(t)\to (0,1)$ when $t\to -b_{\lambda_0}$.
Moreover, when $\lambda\neq\lambda_0$, the trajectory $\varphi_{\lambda}$ either crosses the ray
$\{(\theta,z)\in\mathbb R^2:\theta=0\;\text{and}\;z>1\}$ or converges to a
point in $\partial\Omega$ depending on whether $\lambda>\lambda_0$ or 
$1<\lambda<\lambda_0$ (see Figure \ref{Figure2}).

For each $\lambda>1$, consider the curve
$\alpha_{\lambda}:(-b_{\lambda},b_{\lambda})\to\mathbb R^3$ defined by
$\alpha_{\lambda}(t)=(x(t),0,z(t))$, where
\begin{equation}\label{Rot59b}
x(t)=\int_0^t\cos\theta(s)ds,
\end{equation}
and the surface $M_{\lambda}$ of $\mathbb R^3$ obtained
by the rotation of the image of $\alpha_{\lambda}$ around the
$x$-axis. As we have seen in Section \ref{convexprofcurve}, the length of the 
shape
operator of $M_{\lambda}$ equals 1 at every point. The detailed classification 
of the surfaces $M_\lambda$ reads:

\begin{theorem}\label{surfaces}
	Let $M_{\lambda}$ be as above.
	
	\vskip5pt
	
	\noindent(i) If $\lambda>\lambda_0$ then $M_{\lambda}$ is a
	complete $C^{\infty}$-surface. Moreover, $M_{\lambda}$ is periodic
	and has self-intersections.
	
	\vskip5pt
	
	\noindent(ii) $M_{\lambda_0}$ is incomplete, but it can be extended in infinite many
	ways to a complete $C^3$-surface satisfying $|A|\equiv 1$. Any
	such extension has self-intersections.
	
	\vskip5pt
	
	\noindent(iii) $M_{\sqrt{2}}$ is the sphere with center at
	$(-\sqrt{2},0,0)$ and radius $\sqrt{2}$ (minus two points).
	
	\vskip5pt
	
	\noindent(iv) If $\sqrt{2}<\lambda<\lambda_0$ or
	$1<\lambda<\sqrt{2}$ then $M_{\lambda}$ is incomplete and cannot
	be extended to a surface with $|A|\equiv 1$.
\end{theorem}

Concerning the Gaussian curvature of the surfaces obtained in the above theorem, we observe that
the only surfaces with positive Gaussian curvature are the surfaces $M_{\lambda}$ 
with $1<\lambda\leq \sqrt{2}$. For all the others, the Gaussian curvature 
changes the signal. 

	\vskip10pt

Theorem \ref{surfaces} is a direct consequence of the following result:
\begin{theorem}\label{curves}
	Let $\alpha_{\lambda}:(-b_{\lambda},b_{\lambda})\to\mathbb R^3$ be as 
	above.
	
	\vskip5pt
	
	\noindent(i) If $\lambda>\lambda_0$ then $b_{\lambda}=+\infty$ and
	$\alpha_{\lambda}$ is of class $C^{\infty}$. Moreover,
	$\alpha_{\lambda}$ is periodic and has self-intersections.
	
	\vskip5pt
	
	\noindent(ii) $b_{\lambda_0}<+\infty$ and $\alpha_{\lambda_0}$ can
	be extended in infinite many ways to a profile curve of class
	$C^3$ defined on $\mathbb R$. Any such extension has
	self-intersections.
	
	\vskip5pt
	
	\noindent(iii) $\alpha_{\sqrt{2}}$ is a parametrization by arc
	length of the semicircle in the $xz$-plane with center at 
	$(-\sqrt{2},0,0)$ 
	and
	radius $\sqrt{2}$.
	
	\vskip5pt
	
	\noindent(iv) If $\sqrt{2}<\lambda<\lambda_0$ or
	$1<\lambda<\sqrt{2}$ then $b_{\lambda}<+\infty$ and
	$\alpha_{\lambda}$ cannot be extended to a profile curve defined
	on an open interval containing $(-b_{\lambda},b_{\lambda})$ properly.
\end{theorem}

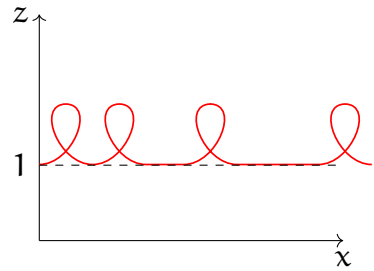
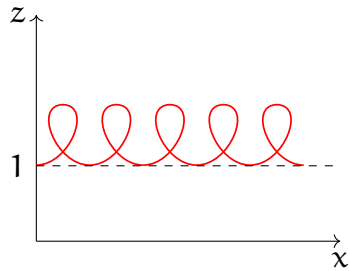
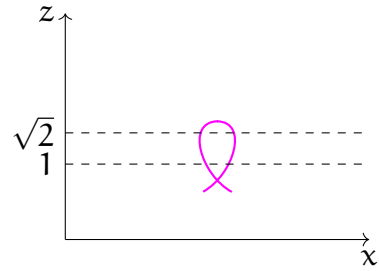
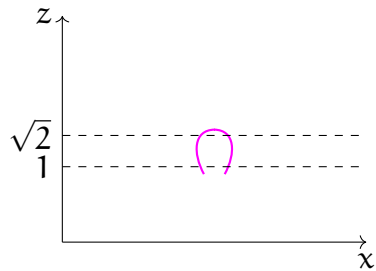
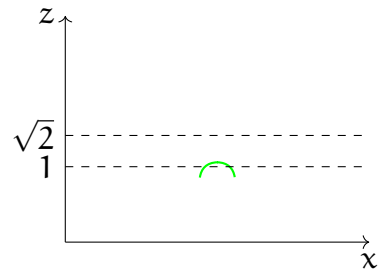
\begin{figure}
	\begin{subfigure}[b]{0.4\linewidth}
		\centering
		\begin{tikzpicture}
		\coordinate (0) at (0,0);
		\coordinate (x) at (4,0);
		\coordinate (y) at (0,3);
		\draw[->] (0)--(x)node[below]{$x$};
		\draw[->] (0)--(y)node[left]{$z$};
		\node[anchor=south west] at (-.15,1.5) 
		{\includegraphics[scale=0.8]{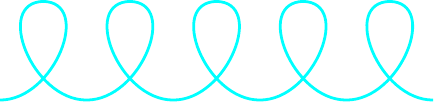}};
		\draw[dashed] (0,1)node[left]{$1$}--(4,1);
		\node[left] at (0,0) {\phantom{$\sqrt{2}$}};
		\end{tikzpicture}
		\caption{$\alpha_{\lambda}$ with $\lambda>\lambda_0$ }
		\vspace{4ex}
	\end{subfigure}
	
	\begin{subfigure}[b]{0.4\linewidth}
		\centering
		\begin{tikzpicture}
		\coordinate (0) at (0,0);
		\coordinate (x) at (4,0);
		\coordinate (y) at (0,3);
		\draw[->] (0)--(x)node[below]{$x$};
		\draw[->] (0)--(y)node[left]{$z$};
		\node[anchor=south west] at (-.15,.85) 
		{\includegraphics[scale=0.8]{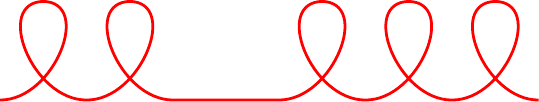}};
		\draw[dashed] (0,1)node[left]{$1$}--(4,1);
		\node[left] at (0,0) {\phantom{$\sqrt{2}$}};
		\end{tikzpicture}
		\caption{$\alpha_{\lambda}$ with $\lambda=\lambda_0$}
		\label{Figb}
		\vspace{4ex}
	\end{subfigure}
	\begin{subfigure}[b]{0.4\linewidth}
		\centering
		\begin{tikzpicture}
		\coordinate (0) at (0,0);
		\coordinate (x) at (4,0);
		\coordinate (y) at (0,3);
		\draw[->] (0)--(x)node[below]{$x$};
		\draw[->] (0)--(y)node[left]{$z$};
		\node[anchor=south west] at (-.15,.85) 
		{\includegraphics[scale=0.8]{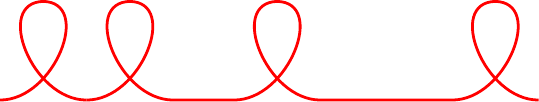}};
		\draw[dashed] (0,1)node[left]{$1$}--(4,1);
		\node[left] at (0,0) {\phantom{$\sqrt{2}$}};
		\end{tikzpicture}
		\caption{$\alpha_{\lambda}$ with $\lambda=\lambda_0$}
		\vspace{4ex}
	\end{subfigure}
	\begin{subfigure}[b]{0.4\linewidth}
		\centering
		\begin{tikzpicture}
		\coordinate (0) at (0,0);
		\coordinate (x) at (4,0);
		\coordinate (y) at (0,3);
		\draw[->] (0)--(x)node[below]{$x$};
		\draw[->] (0)--(y)node[left]{$z$};
		\node[anchor=south west] at (-.15,.85) 
		{\includegraphics[scale=0.8]{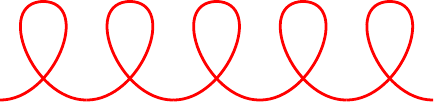}};
		\draw[dashed] (0,1)node[left]{$1$}--(4,1);
		\node[left] at (0,0) {\phantom{$\sqrt{2}$}};
		\end{tikzpicture}
		\caption{$\alpha_{\lambda}$ with $\lambda=\lambda_0$}
		\vspace{4ex}
	\end{subfigure}
	\begin{subfigure}[b]{0.4\linewidth}
		\centering
		\begin{tikzpicture}
		\coordinate (0) at (0,0);
		\coordinate (x) at (4,0);
		\coordinate (y) at (0,3);
		\draw[->] (0)--(x)node[below]{$x$};
		\draw[->] (0)--(y)node[left]{$z$};
		\node[anchor=center] at (2,1.1) 
		{\includegraphics[scale=1]{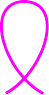}};
		\draw[dashed] (0,1)node[left]{$1$}--(4,1);
		\draw[dashed] 
		(0,1.4142)node[left]{$\sqrt{2}$}--(4,1.4142);
		\end{tikzpicture}
		\caption{$\alpha_\lambda$ with 
			$\sqrt{2}<\lambda<\lambda_0$}
		\vspace{4ex}
	\end{subfigure}
	\begin{subfigure}[b]{0.4\linewidth}
		\centering
		\begin{tikzpicture}
		\coordinate (0) at (0,0);
		\coordinate (x) at (4,0);
		\coordinate (y) at (0,3);
		\draw[->] (0)--(x)node[below]{$x$};
		\draw[->] (0)--(y)node[left]{$z$};
		\node[anchor=center] at (2,1.2) 
		{\includegraphics[scale=1]{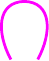}};
		\draw[dashed] (0,1)node[left]{$1$}--(4,1);
		\draw[dashed] 
		(0,1.4142)node[left]{$\sqrt{2}$}--(4,1.4142);
		\end{tikzpicture}
		\caption{$\alpha_\lambda$ with 
			$\sqrt{2}<\lambda<\lambda_0$}
		\vspace{4ex}
	\end{subfigure}
	\begin{subfigure}[b]{0.4\linewidth}
		\centering
		\begin{tikzpicture}
		\coordinate (0) at (0,0);
		\coordinate (x) at (4,0);
		\coordinate (y) at (0,3);
		\draw[->] (0)--(x)node[below]{$x$};
		\draw[->] (0)--(y)node[left]{$z$};
		\node[anchor=north] at (2,1.22) 
		{\includegraphics[scale=01]{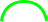}};
		\draw[dashed] (0,1)node[left]{$1$}--(4,1);
		\draw[dashed] 
		(0,1.4142)node[left]{$\sqrt{2}$}--(4,1.4142);
		\end{tikzpicture}
		\caption{$\alpha_\lambda$ with $1<\lambda<\sqrt{2}$}
		\vspace{4ex}
	\end{subfigure}
	\caption{Graphic representation of the curves $a_\lambda$}\label{Figure3}
\end{figure}

\newpage

\begin{proof} 
(i) From
$\lambda>\lambda_0$ and the discussion in the beginning of this
section one infers that there exists $t_0>0$ such that
$\theta(-t_0)=0$. Applying Lemma \ref{Simetrico} with $n=0$ and
$n=1$ one concludes that $b_{\lambda}=+\infty$. Being the
trajectory of a vector field of class $C^{\infty}$,
$\varphi_{\lambda}$, and hence $\alpha_{\lambda}$, is of class
$C^{\infty}$.

We will now prove that $\alpha_{\lambda}$ is periodic. Since, by
\Eqref{VectorField} and Lemma \ref{Simetrico}, the maps
$t\in\mathbb R\mapsto\varphi_{\lambda}(t+2t_0)$ and $t\in\mathbb
R\mapsto(\theta(t)+2\pi,z(t))$ are both trajectories of $X$
passing through $(3\pi,\lambda)$, one has
\begin{equation}\label{Rot30}
\varphi_{\lambda}(t+2t_0)=\big(\theta(t)+2\pi,z(t)\big),\;\;\;t\in\mathbb
R.
\end{equation}
On the other hand, by \Eqref{Rot59b} one has
\begin{equation}\label{Rot48b}
\begin{aligned}
x(t+2t_0)=&\ \int_{0}^{t+2t_0}\cos\theta(s)ds\\
=&\int_{0}^{t}\cos\theta(s)ds
+\int_{t}^{t+2t_0}\cos\theta(s)ds\\
=&\ x(t)+\int_{t}^{2t_0}\cos\theta(s)ds+\int_{2t_0}^{2t_0+t}\cos\theta(s)ds\\
=&\ x(t)+\int_{t}^{2t_0}\cos\theta(s)ds+\int_{0}^{t}\cos\theta(s+2t_0)ds.
\end{aligned}
\end{equation}
Since $\theta(s+2t_0)=\theta(s)+2\pi$ by \Eqref{Rot30}, it follows
that
\begin{equation}\label{Rot48c}
\begin{aligned}
x(t+2t_0)=&\ x(t)+\int_{t}^{2t_0}\cos\theta(s)ds
+\int_{0}^{t}\cos\theta(s)ds\\
=&\ x(t)+\int_{0}^{2t_0}\cos\theta(s)ds=\ x(t)+x(2t_0),\;\;\;t\in\mathbb R.
\end{aligned}
\end{equation}
Since $z(t+2t_0)=z(t)$ for all $t\in\mathbb R$ by \Eqref{Rot30},
the curve $\alpha_{\lambda}$ is periodic.

To complete the proof of (i), it remains to show that
$\alpha_{\lambda}$ is non-embedded. In fact, we will show that the
restriction of $\alpha_{\lambda}$ to the interval $(-t_0,t_0)$ has
already self-intersections. For that observe first that, since
$\theta(-t_0)=0$, $\theta(0)=\pi$ and $\theta'>0$, the function
$\theta=\theta|_{[-t_0,0]}$ is a diffeomorphism from $[-t_0,0]$ to
$[0,\pi]$. In particular, there exists a unique $t_1\in(0,t_0)$
such that $\theta(-t_1)=\pi/2$. Let $\xi:[-t_0,-t_1]\to[-t_1,0]$ be
defined by
\begin{equation}\label{Rot35}
\xi(t)=\theta^{-1}(\pi-\theta(t)).
\end{equation}
Clearly, $\xi$ is a diffeomorphism, $\xi(-t_0)=0$ and
$\xi(-t_1)=-t_1$. Moreover,
\begin{equation}\label{Rot36}
\xi'(t)=-\frac{\theta'(t)}{\theta'(\xi(t))}\,,\;\;\;t\in[-t_0,-t_1],
\end{equation}
and
\begin{equation}\label{Rot36a}
\cos\theta(\xi(t))=\cos(\pi-\theta(t))=-\cos\theta(t),\;\;\;t\in[-t_0,-t_1].
\end{equation}
By \Eqref{sistema} and \eqref{Rot36a},
\begin{equation}\label{Rot37}
\theta'(\xi(t))=\sqrt{1-\frac{\cos^2\theta(\xi(t))}{z^2(\xi(t))}}
=\sqrt{1-\frac{\cos^2\theta(t)}{z^2(\xi(t))}}.
\end{equation}
Since, by \Eqref{sistema},
\begin{equation*}
z(\xi(t))-z(t)=\int_t^{\xi(t)}z'(s)ds
=\int_t^{\xi(t)}\sin\theta(s)ds>0,\;\;\;t\in[-t_0,-t_1),
\end{equation*}
one then has
\begin{equation}\label{Rot39}
\theta'(\xi(t))>\sqrt{1-\frac{\cos^2\theta(t)}{z^2(t)}}
=\theta'(t),\;\;\;t\in[-t_0,-t_1).
\end{equation}

Using the informations collected above, we will now compare the
values of $x(t)$ for $t=-t_0$, $t=-t_1$ and $t=0$. Since
$\pi/2\leq\theta(t)\leq\pi$ for $t\in[-t_1,0]$, from \Eqref{Rot59b}
we obtain
\begin{equation}\label{Rot40a}
x(-t_1)=\int_0^{-t_1}\cos\theta(s)ds=-\int_{-t_1}^{0}\cos\theta(s)ds>0=x(0).
\end{equation}
On the other hand, by \Eqref{Rot36}, \eqref{Rot36a} and \eqref{Rot39}, one has
\begin{equation}\label{Rot41}
\begin{aligned}
\int_{-t_1}^0\cos\theta(t)dt =& \ \int_{-t_1}^{-t_0}\cos\theta(\xi(s))\xi'(s)ds
=\int_{-t_0}^{-t_1}\frac{\cos\theta(\xi(s))\theta'(s)}{\theta'(\xi(s))}ds\\
>&\ \int_{-t_0}^{-t_1}\cos\theta(\xi(s))ds=-\int_{-t_0}^{-t_1}\cos\theta(s)ds.
\end{aligned}
\end{equation}
Hence, by \Eqref{Rot59b} and inequality above,
\begin{equation}\label{Rot42}
x(-t_0)=-\int_{-t_0}^0\cos\theta(t)dt
=-\int_{-t_0}^{-t_1}\cos\theta(t)dt-\int_{-t_1}^0\cos\theta(t)dt<0.
\end{equation}

The curve $\alpha_{\lambda}$ is symmetric with respect to the line
$x=x(0)=0$. Indeed, by Lemma \ref{Simetrico} one has
\begin{equation}\label{Rot43}
\theta(t)=2\pi-\theta(-t),\;\;z(t)=z(-t),\;\;\;t>0,
\end{equation}
and so
\begin{equation}\label{Rot44}
\begin{aligned}
x(t)=&\ \int_{0}^{t}\cos\theta(s)ds=\int_{0}^{t}\cos\theta(-s)ds\\
=&\ -\int_{0}^{-t}\cos\theta(s)ds=-x(-t),
\end{aligned}
\end{equation}
for every $t>0$.

Since $x'>0$ on $(-t_0,-t_1)$ and, by \Eqref{Rot40a} and
\eqref{Rot42}, $x(-t_0)<0<x(-t_1)$, there exists a unique
$t_2\in(t_1,t_0)$ such that $x(-t_2)=0=x(0)$. Then, by \Eqref{Rot44},
\begin{equation}\label{Rot48}
x(t_2)=-x(-t_2)=0=x(-t_2).
\end{equation}
Since $z(t_2)=z(-t_2)$ by \Eqref{Rot43}, it follows that
$\alpha_{\lambda}(t_2)=\alpha_{\lambda}(-t_2)$. Hence, the
restriction of $\alpha_{\lambda}$ to the interval $(-t_0,t_0)$ has
a self-intersec-tion.

\vskip10pt

\noindent(ii) We begin by showing that $b_{\lambda_0}<+\infty$. 
Let $t_1\in(-b_{\lambda_0},0)$ such that $\theta(t_1)=\pi/2$. For
every $t\in(-b_{\lambda_0},t_1]$ we have
\begin{equation}\label{Finito7}
\begin{aligned}
t_1-t=& \ \int_t^{t_1}1ds=\int_{\theta(t)}^{\pi/2}(\theta^{-1})'(u)du
=\int_{\theta(t)}^{\pi/2}\frac{1}{\theta'(\theta^{-1}(u))}du\\
=& \ 
\int_{\theta(t)}^{\pi/2}\frac{z}{\sqrt{z^2-\cos^2\theta}}
(\theta^{-1}(u))du\\
\leq&
\lambda_0\int_{\theta(t)}^{\pi/2}\frac{1}{\sqrt{z-\cos\theta}}
(\theta^{-1}(u))\,du,
\end{aligned}
\end{equation}
where in the last inequality we used the fact that
$1<z\circ\theta^{-1}<\lambda_0$ and the cosine function is nonnegative on $(0,\pi/2]$.

\begin{claim}
There is $C_1>0$ such that
\begin{equation}\label{Finito6}
\frac{1}{\sqrt{z(t)-\cos\theta(t)}}\leq\frac{C_1}{\sin^{\frac{3}{4}}\theta(t)},\;\;\;\;t\in(-b_{\lambda_0},t_1].
\end{equation}
\end{claim}

Indeed, since
$\theta(t)\to 0$ and $z(t)\to 1$ when $t\to -b_{\lambda_0}$, from 
\Eqref{sistema} we obtain

\begin{equation}\label{ThetaLinhavaiparazero} 
    \lim\limits_{t\to -b_{\lambda_0}}\theta'(t)=0.
\end{equation}
Then, again by \Eqref{sistema},
\begin{equation}\label{Finito1}
\begin{aligned}
\lim_{t\to
	-b_{\lambda_0}}\frac{\sin^2\theta(t)}{\theta'(t)^2}=& \ \lim_{t\to
	-b_{\lambda_0}}\frac{z^2(t)}{z(t)+\cos\theta(t)}\,
\frac{\sin^2\theta(t)}{z(t)-\cos\theta(t)}\\
=&\ \frac{1}{2}\lim_{t\to
	-b_{\lambda_0}}\frac{\sin^2\theta(t)}{z(t)-\cos\theta(t)}\\
=& \ \frac{1}{2}\lim_{t\to-b_{\lambda_0}}
\frac{2\theta'(t)\sin\theta(t)\cos\theta(t)}{\sin\theta(t)(1+\theta'(t))}\\
=&\ \frac{1}{2}\lim_{t\to
	-b_{\lambda_0}}\frac{2\theta'(t)\cos\theta(t)}{1+\theta'(t)}=0,
\end{aligned}
\end{equation}
and so
\begin{equation}\label{Finito2}
\begin{aligned}
\lim_{t\to -b_{\lambda_0}}\frac{z(t)-\cos\theta(t)}{\sin\theta(t)}
=& \ \lim_{t\to -b_{\lambda_0}}
\frac{\sin\theta(t)(1+\theta'(t))}{\theta'(t)\cos\theta(t)}\\
=& \ \lim_{t\to
	-b_{\lambda_0}}\frac{\sin\theta(t)}{\theta'(t)}=0.
\end{aligned}
\end{equation}
From \Eqref{sistema} and the above equality one obtains
\begin{equation}\label{Finito4}
\lim_{t\to
	-b_{\lambda_0}}\frac{\theta'(t)^2}{\sin\theta(t)}=\lim_{t\to
	-b_{\lambda_0}}\frac{z(t)+\cos\theta(t)}{z^2(t)}\,\frac{z(t)-\cos\theta(t)}{\sin\theta(t)}=0.\nonumber
\end{equation}
Therefore,
\begin{equation}\label{Finito3} 
\begin{aligned}
\lim_{t\to 
	-b_{\lambda_0}}\frac{\sin^{\frac{3}{2}}\theta(t)}{z(t)-\cos\theta(t)}
=& \ \lim_{t\to-b_{\lambda_0}}\frac{\frac{3}{2}\theta'(t)
	\sin^{\frac{1}{2}}\theta(t)\cos\theta(t)}{\sin\theta(t)(1+\theta'(t))}\\
=& \ \frac{3}{2}\lim_{t\to-b_{\lambda_0}}\frac{\theta'(t)}{\sin^{\frac{1}{2}}
	\theta(t)}=0,
\end{aligned}
\end{equation}
and the claim follows.

From \Eqref{Finito7} and \eqref{Finito6}, we obtain
\begin{equation}\label{Finito8}
t_1-t\leq\lambda_0C_1\int_{\theta(t)}^{\pi/2}\sin^{-\frac{3}{4}}udu,\;\;\;
-b_{\lambda_0}<t<t_1.
\end{equation}
Setting $C_2=\inf\left\{\sin u/u:u\in(0,\frac{\pi}{2}]\right\}$, one has
\begin{equation*}\label{Finito9}
\sin^{-\frac{3}{4}}u\leq
C_2^{-\frac{3}{4}}u^{-\frac{3}{4}},\;\;\;\;u\in(0,\frac{\pi}{2}].
\end{equation*}
Using this information in \Eqref{Finito8}, we obtain
\begin{equation}\label{Finito10}
\begin{aligned}
t_1-t\leq\lambda_0C_1C_2^{-\frac{3}{4}}\int_{\theta(t)}^{\pi/2}
u^{-\frac{3}{4}}du
=& \ 4\lambda_0C_1C_2^{-\frac{3}{4}}u^{\frac{1}{4}}\Big|_{\theta(t)}^{\pi/2}\\
=&\ 4 \lambda_0C_1C_2^{-\frac{3}{4}}\left[\left(\pi/2\right)^{\frac{1}{4}}
-\theta(t)^{\frac{1}{4}}\right]\\
<& \ 4\lambda_0C_1C_2^{-\frac{3}{4}}(\pi/2)^{\frac{1}{4}},
\end{aligned}
\end{equation}
for every $t\in(-b_{\lambda_0},t_1)$. Therefore,
$b_{\lambda_0}<+\infty$.

In order to prove that $\alpha_{\lambda_0}$ can be extended to a
profile curve of class $C^3$ defined on $\mathbb R$, we need to
evaluate the limits of $\theta''$ and $\theta'''$ when
$t\to -b_{\lambda_0}$.  
From \Eqref{sistema} one obtains, after some work,
\begin{equation}\label{Rot52}
\theta''(t)=\frac{\sin\theta(t)}{z(t)\theta'(t)}
+\frac{\sin\theta(t)\cos\theta(t)}{z^2(t)}
-\frac{\theta'(t)\sin\theta(t)}{z(t)}.
\end{equation}
Since $\theta(t)\to 0$, $z(t)\to 1$ and
$\theta'(t)\to 0$ when $t\to -b_{\lambda_0}$, it follows from \Eqref{Finito2} and \eqref{Rot52} that
\begin{equation}\label{Rot53}
\lim_{t\to -b_{\lambda_0}}\theta''(t)=\lim_{t\to
	-b_{\lambda_0}}\frac{\sin\theta(t)}{\theta'(t)}=0.
\end{equation}
As to $\lim\limits_{t\to -b_{\lambda_0}}\theta'''(t)$,
observe that, since $z(t)\to 1$, $\theta(t)\to 0$, $\theta'(t)\to
0$ and $\theta''(t)\to 0$, the derivatives of either of the last
two terms on the right hand side of \Eqref{Rot52} goes to zero
when $t\to -b_{\lambda_0}$. Therefore,
\begin{equation}\label{Rot55}
\begin{aligned}
\lim_{t\to -b_{\lambda_0}}\theta'''(t)=&\ \lim_{t\to
	-b_{\lambda_0}}\left(\frac{\sin\theta}{z\theta'}\right)'(t)\\
=&\ \lim_{t\to
	-b_{\lambda_0}}\left\{-\frac{\sin\theta}{z^2}\frac{\sin\theta}{\theta'}
+\frac{1}{z}\left(\frac{\sin\theta}{\theta'}\right)'\right\}(t),\\
=&\ \lim_{t\to
	-b_{\lambda_0}}\left(\frac{\sin\theta}{\theta'}\right)'(t),
\end{aligned}
\end{equation}
where in the last equality we used \Eqref{Finito1}. Since, by
\Eqref{sistema},
\begin{equation}\label{Rot55a}
\frac{\sin\theta}{\theta'}=\frac{z}{\sqrt{z+\cos\theta}}\,
\frac{\sin\theta}{\sqrt{z-\cos\theta}}\,,
\end{equation}
one has
\begin{equation}\label{Rot57}
\begin{aligned}
\left(\frac{\sin\theta}{\theta'}\right)'=&\ \left(
\frac{\sin\theta}{\sqrt{z+\cos\theta}}
-\frac{z\sin\theta(1-\theta')}{2(z+\cos\theta)^{3/2}}\right)
\frac{\sin\theta}{\sqrt{z-\cos\theta}}\\
&\ +\cos\theta-\frac{z(1+\theta')}{2\sqrt{z+\cos\theta}}\,
\frac{\sin^2\theta}{(z-\cos\theta)^{3/2}}.
\end{aligned}
\end{equation}
The first term on the right hand side goes to zero by
\Eqref{Finito1} and \eqref{Rot55a}. Hence,
\begin{equation}\label{Rot58}
\lim_{t\to
	-b_{\lambda_0}}\left(\frac{\sin\theta}{\theta'}\right)'=1-\frac{1}{2\sqrt{2}}\lim_{t\to
	-b_{\lambda_0}}\frac{\sin^2\theta}{(z-\cos\theta)^{3/2}}.
\end{equation}
Using \Eqref{sistema} again, one obtains
$$
\lim_{t\to
	-b_{\lambda_0}}\frac{\sin^2\theta}{(z-\cos\theta)^{3/2}}=\lim_{t\to
	-b_{\lambda_0}}\frac{4\cos\theta\sqrt{z+\cos\theta}}{3z(1+\theta')}=\frac{4\sqrt{2}}{3}.
$$
It now follows from \Eqref{Rot55}, \Eqref{Rot58} and the above
equality that
\begin{equation}\label{Rot59}
\lim_{t\to
	-b_{\lambda_0}}\theta'''(t)=1-\frac{1}{2\sqrt{2}}\frac{4\sqrt{2}}{3}=\frac{1}{3}.
\end{equation}

It is possible to extend $\alpha_{\lambda_0}$ gluing together copies of 
$\alpha_{\lambda_0}$. By \Eqref {ThetaLinhavaiparazero}, \eqref{Rot53} and 
\eqref{Rot59}, this extension is (at least) $C^4$ (recall that if a profile 
curve is of class $C^s$ then its corresponding angle function is of class 
$C^{s-1}$). Another way to extend $\alpha_{\lambda_0}$ is gluing together 
copies of $\alpha_{\lambda_0}$ and horizontal segments with any length and 
with height equal to 1. By the same equations, these extensions are $C^3$ but 
not $C^4$ (see Figure \ref{Figure3}, items (B), (C) and (D), for a sample of 
these 
extensions).


To complete the proof of (ii), it remains to show that any extension of 
$\alpha_{\lambda_0}$ is non-embedded. But clearly this follows from the fact 
that $\alpha_{\lambda_0}$ has a self-intersection, which in turn can be 
proved as in (i) (with $b_{\lambda_0}$ playing the role of $t_0$).

\vskip10pt

\noindent(iii) As can be easily seen, the curve
\begin{equation*}\label{Rot70}
\psi(t)=(\pi+t/\sqrt{2},\sqrt{2}\cos(t/\sqrt{2})),\;\;\;
t\in(-\sqrt{2}\pi/2,\sqrt{2}\pi/2),
\end{equation*}
is a trajectory of $X$. Since $\psi(0)=(\pi,\sqrt{2})$, one has 
$\varphi_{\sqrt{2}}=\psi$.
Then, by \Eqref{Rot59b},
$$
\alpha_{\sqrt{2}}(t)=(-\sqrt{2}\sin(t/\sqrt{2})-\sqrt{2},
\sqrt{2}\cos(t/\sqrt{2})),\;\;\;t\in(-\frac{\sqrt{2}\pi}{2},
\frac{\sqrt{2}\pi}{2}),
$$
which is a parametrization by arc length of the portion of the circle in the 
$xz$-plane with
center $(-\sqrt{2},0,0)$ and radius $\sqrt{2}$ that is above the $x$-axis.

\vskip10pt

\noindent(iv) Let $(\theta_0,z_0)=\lim_{t\to
	-b_{\lambda}}\varphi_{\lambda}(t)$. Since $\lambda\neq\sqrt{2}$
and, by item (iii), $\varphi_{\sqrt{2}}(t)\to(\pi/2,0)$ when $t\to
-b_{\sqrt{2}}=-\sqrt{2}\pi/2$, it follows from Lemma
\ref{UnicLimPoint} that $0<z_0<1$ and either $0<\theta_0<\pi/2$ or
$\pi/2<\theta_0<\pi$.

Suppose, by contradiction, that $b_{\lambda}=+\infty$. Since
$\theta(t)\to\theta_0$ when $t\to -b_{\lambda}$ and either
$0<\theta_0<\pi/2$ or $\pi/2<\theta_0<\pi$, there exist
$\varepsilon>0$ and $t_1\in\mathbb R$ such that
$$
\sin\theta(t)>\varepsilon,\;\;\;t\leq t_1.
$$
Then, by \Eqref{Rot1},
$$
z(t_1)-z_0>z(t_1)-z(t)=\int_t^{t_1}\sin\theta(s)ds>\varepsilon(t_1-t),
$$
for every $t<t_1$, a contradiction. Hence, $b_{\lambda}<\infty$.

Suppose, by contradiction, that $\alpha_{\lambda}$ can be extended
to a profile curve $\widetilde\alpha_{\lambda}(t)=(\widetilde
x(t),0,\widetilde z(t)),\;t\in(a,b),$ where $a<-b_{\lambda}$, say.
Let $\widetilde\theta:(a,b)\to\mathbb R$ be a function satisfying
$$
\widetilde\alpha_{\lambda}'(t)=(\cos\widetilde\theta(t),0,\sin\widetilde\theta
(t)),\;t\in(a,b).
$$
Since $\widetilde\alpha_{\lambda}'(t)=\alpha_{\lambda}'(t)$ for
all $t\in(-b_{\lambda},b_{\lambda})$, and, by \Eqref{Rot59b},
$$
\alpha_{\lambda}'(t)=(x'(t),0,z'(t))=(\cos\theta(t),0,\sin\theta(t)),
\;\;\;t\in(-b_{\lambda},b_{\lambda}),
$$
one concludes that $\theta$ and the restriction of
$\widetilde\theta$ to the interval $(-b_{\lambda},b_{\lambda})$
differ by an integer multiple of $2\pi$. Assuming without loss of
generality that
$\theta=\widetilde\theta|_{(-b_{\lambda},b_{\lambda})}$, one has
$$
\widetilde\theta'(-b_{\lambda})=\lim_{t\to
	-b_{\lambda}}\widetilde\theta'(t)=\lim_{t\to
	-b_{\lambda}}\theta'(t)=0
$$
and
$$
\widetilde z(-b_{\lambda})=\lim_{t\to -b_{\lambda}}\widetilde
z(t)=\lim_{t\to -b_{\lambda}}z(t)=z_0,
$$
which contradicts Lemma \ref{ThetaLinhaZero} (i) as $0<z_0<1$.
\end{proof}

Now we finish the prove of our main theorem.

\vskip5pt

\noindent{\bf Completion of proof of Theorem \ref{Family}:} Let $\alpha_{\lambda}$ and $M_{\lambda}$, $\lambda>1$, be as in the beginning of this section, and let $\mathcal F_1=\{M_{\lambda}:\lambda>\lambda_0\}$. Let $\mathcal A$ be the set of extensions of $\alpha_{\lambda_0}$ constructed in the proof of Theorem \ref{curves} (ii), and $\mathcal F_2$ the (infinite) family of surfaces obtained by the rotation of the images of the curves in $\mathcal A$ around the $x$-axis. By Theorem \ref{surfaces} (i), the surfaces in $\mathcal F_1$ have the properties  stipulated in the statement of the theorem. That the surfaces in $\mathcal F_2$ also meet the stipulated conditions is a consequence of the proof of Theorem \ref{curves} (ii). This establishes the first part of the theorem.

Let $M$ be a complete rotational surface of class $C^2$ satisfying $|A|\equiv 
1$. Applying a rigid motion if necessary, we can assume that its axis of revolution is 
the $x$-axis and that its profile curve $\mathcal C$ is contained in the $xz$-plane. Let  $\alpha(t)=(x(t),0,z(t))$, $t\in(a,b)$, be a 
parametrization of $\mathcal C$ such that $z(t)>0$ and 
$||\alpha'(t)||=1$ for all $t\in(a,b)$. Let $\theta:(a,b)\to\mathbb R$ be a 
$C^1$-function such that 
$$
(x'(t),0,z'(t))=(\cos\theta(t),0,\sin\theta(t)),\;\;\;t\in(a,b).
$$
As we have seen in Section \ref{convexprofcurve}, the function $\theta$ 
satisfies
$$
\theta'^2(t)+\frac{\cos^2\theta(t)}{z^2(t)}=1,\;\;\;t\in(a,b).
$$ 

If $\theta$ is constant, by the above equality one has that $z$ is constant. 
Then $\sin\theta=z'=0$ and so $z=|\cos\theta|\equiv 1$. Being complete, $M$ 
is then a right circular cylinder of radius 1. 

Suppose now that $\theta$ is not constant. Since $\theta$ is monotone by Proposition \ref{monotona}, reparametrizing $\alpha$ if necessary we can assume that $\theta'(t) \geq 0$ for all $t\in(a,b)$. Let 
$t_0\in(a,b)$ such that $\theta'(t_0)>0$. Without loss of generality, we can 
assume that $0\leq\theta(t_0)<2\pi$. Let $(a_1,b_1)\subset(a,b)$ be the 
maximal interval containing $t_0$ on which $\theta'>0$. As we have seen in 
Section \ref{convexprofcurve}, the map 
$t\in(a_1,b_1)\mapsto\varphi(t):=(\theta(t),z(t))$ is an integral curve of the 
vector field $X$ defined by \Eqref{VectorField}. By the results in Section 
\ref{phasefundvectfield}, there is a unique trajectory 
$\varphi_{\lambda}:(-b_{\lambda},b_{\lambda})\to\Omega$ in the family $\{\varphi_{\lambda}\}_{\lambda>1}$ that passes 
through $(\theta(t_0),z(t_0))$. Let $s_0\in(-b_{\lambda},b_{\lambda})$ 
such that $\varphi_{\lambda}(s_0)=(\theta(t_0),z(t_0))=\varphi(t_0)$. 
Assuming without loss of generality that $t_0=s_0$, it follows from the 
maximality of $\varphi_{\lambda}$ that 
$\varphi=\varphi_{\lambda}|_{(a_1,b_1)}$. Write 
$\varphi_{\lambda}=(\theta_{\lambda},z_{\lambda})$ and consider its associated 
profile curve 
$t\in(-b_{\lambda},b_{\lambda})\mapsto\alpha_{\lambda}(t)=(x_{\lambda}(t),0,z_{\lambda}(t))$
 (\cf the beginning of this section). Since 
$$
x_{\lambda}'(t)=\cos\theta_{\lambda}(t)=\cos\theta(t)=x'(t),\;\;\;t\in(a_1,b_1),
$$
it holds that
\begin{equation}\label{new}
\alpha(t)=\alpha_{\lambda}(t)+(d,0,0),\;\;\;t\in(a_1,b_1),
\end{equation}
for some $d\in\mathbb R$.

\vskip5pt

\noindent{\bf Claim.} $a\leq -b_{\lambda}$ and $b\geq b_{\lambda}$.

\vskip5pt

Assuming, by contradiction, that $-b_{\lambda}<a$, one has $-b_{\lambda}<a_1$. Then, since $\theta_{\lambda}'>0$ on $(-b_{\lambda},b_{\lambda})$ 
and $\theta|_{(a_1,b_1)}=\theta_{\lambda}|_{(a_1,b_1)}$, 
$$
\lim_{t\to a_1}\theta'(t)=\lim_{t\to a_1}\theta_{\lambda}'(t)=\theta_{\lambda}'(a_1)>0.
$$ 
From the above inequality and definition of $(a_1,b_1)$ one obtains $a=a_1$. It now follows from \Eqref{new} that $\alpha$ can be extended to an interval containing $(a,b)$ properly, contradicting the completeness of $M$. This contradiction proves that $a\leq -b_{\lambda}$. In the same way, one proves that $b\geq b_{\lambda}$.

It follows from the Claim that $a_1=-b_{\lambda}$. Indeed, if we had $-b_{\lambda}<a_1$, reasoning as above one would obtain $\lim_{t\to a_1}\theta'(t)>0$. On the other hand, from $a<a_1$ one would obtain  $\lim_{t\to a_1}\theta'(t)=0$, a contradiction. In the same manner, one proves that $b_1=b_{\lambda}$. Hence $(a_1,b_1)=(-b_{\lambda},b_{\lambda})$ and, by \Eqref{new},
\begin{equation}\label{new3}
\alpha(t)=\alpha_{\lambda}(t)+(d,0,0),\;\;\;t\in(-b_{\lambda},b_{\lambda}),
\end{equation}
for some $d\in\mathbb R$.

From \Eqref{new3} one obtains that either $\lambda\geq\lambda_0$ or $\lambda=\sqrt{2}$. In fact, if we had $1<\lambda<\sqrt{2}$ or $\sqrt{2}<\lambda<\lambda_0$, from Theorem \ref{curves} (iv) we would obtain that $M$ is a translation of $M_{\lambda}$, and so $M$ would be incomplete, contradicting the hypothesis.

In the case $\lambda>\lambda_0$, it follows from \Eqref{new3} and Theorem \ref{curves} (i) that $\alpha(t)=\alpha_{\lambda}(t)+(d,0,0)$ for all $t\in\mathbb R$, and therefore $M=M_{\lambda}$ (up to translation). 

In the case $\lambda=\sqrt{2}$, it follows from \Eqref{new3} and Theorem \ref{curves} (iii) that $\alpha(t)=\alpha_{\sqrt{2}}(t)+(d,0,0),\;t\in(-b_{\sqrt{2}},b_{\sqrt{2}})$. Since $M$ is complete, one concludes that $M$ is, up to translation, the sphere with center at $(-\sqrt{2},0,0)$ and radius $\sqrt{2}$.  

Finally, consider the case $\lambda=\lambda_0$. By \Eqref{new3}, $M$ is an extension of $M_{\lambda_0}$. Since $M$ is complete and, by Theorem \ref{surfaces} (ii), $M_{\lambda_0}$ is incomplete, one has $a<-b_{\lambda_0}$ and $b>b_{\lambda_0}$. We will conclude that, up to congruence, $\alpha$ belongs to the family $\mathcal A$ and so $M\in\mathcal F_2$. For that we can assume that $\theta'$ is not identically zero on $(a,-b_{\lambda_0})\cup(b_{\lambda_0},b)$, for otherwise $z=1$ outside $(-b_{\lambda_0},b_{\lambda_0})$ and the conclusion holds trivially. We claim that for any $s\in(a,-b_{\lambda_0})\cup(b_{\lambda_0},b)$ at which $\theta'(s)>0$, there is an open interval $I$ of length $2b_{\lambda_0}$ containing $s$ such that $\alpha(I)$ differs from the image of $\alpha_{\lambda_0}$ by a horizontal vector. We will prove the claim in the case $b_{\lambda_0}<s<b$ (the proof in the case $a<s<-b_{\lambda_0}$ is analogous). Denote by $(a_2,b_2)\subset(a,b)$ the maximal interval 
 containing $s$ on which $\theta'>0$. By what we have already proved (\cf 
 \Eqref{new3}), $\alpha((a_2,b_2))$ coincides with a horizontal translation of 
 the image of $\alpha_{\lambda}$, for some $\lambda>1$. Since $a_2>-\infty$ and 
 $z(a_2)=1$ (by Lemma \ref{ThetaLinhaZero}(i), since $\theta'(a_2)=0$), we have 
 $\lambda=\lambda_0$ and the claim is proved. Let $I,J$ be subintervals of 
 $(a,b)$ such that $\alpha(I)$ and $\alpha(J)$ are both horizontal translations 
 of the image of $\alpha_{\lambda_0}$. If the distance between $I$ and $J$ is 
 positive but less than $2b_{\lambda_0}$, then, by the previous claim, one has 
 $\theta'=0$, and hence $z=1$, in the interval between $I$ and $J$. It is now 
 clear that the image of $\alpha$ is made up of curves congruent to 
 $\alpha_{\lambda_0}$ and eventually of horizontal segments of height equal to 
 1. Therefore, $\alpha\in\mathcal A$ and so $M\in\mathcal F_2$.\qed

\bibliography{BarFonHar}
\bibliographystyle{amsplain}

\end{document}